%% file: Experimental_Comparison_SONC_SOS_Arxiv1.tex
\newif\ifpictures
\picturestrue

\documentclass[12pt]{amsart}

\usepackage{latex_base}
\usepackage{macros}
\usepackage{tablefootnote}
\input{numbers.tex}

\usepackage{listings}
\setlist[description]{leftmargin=1em,labelindent=0em}

\author{Henning Seidler} 

\address{Henning Seidler, Technische Universit\"at Berlin, Institut f\"ur Mathematik, Stra{\ss}e des 17.~Juni 136, 10623 Berlin,
 Germany\medskip}

\email{seidler@math.tu-berlin.de}

\author{Timo de Wolff}

\address{Timo de Wolff, Technische Universit\"at Berlin, Institut f\"ur Mathematik, Stra{\ss}e des 17.~Juni 136, 10623 Berlin,
 Germany\medskip}

\email{dewolff@math.tu-berlin.de}



\newif\ifcomment
\commentfalse
\commenttrue

\title[An Experimental Comparison of SONC and SOS Certificates for Unconstrained Optimization]{An Experimental Comparison of SONC and SOS Certificates for Unconstrained Optimization}

\begin{document}

\pagestyle{plain}
\pagenumbering{arabic}

\begin{abstract}
	Finding the minimum of a multivariate real polynomial is a well-known hard problem with various applications.
	We present a polynomial time algorithm to approximate such lower bounds via sums of nonnegative circuit polynomials (SONC).
	As a main result, we carry out the first large-scale comparison of SONC, using this algorithm and different geometric programming (GP) solvers, with the classical sums of squares (SOS) approach, using several of the most common semidefinite programming (SDP) solvers.
	SONC yields bounds competitive to SOS in several cases, but using significantly less time and memory. 
	In particular, SONC/GP can handle much larger problem instances than SOS/SDP.
\end{abstract}

\maketitle

\section{Introduction}

Minimizing a given real, multivariate polynomial $f \in \R[\Vector{x}]$ is the fundamental challenge of polynomial optimization.
This problem has countless applications, see e.g., \cite{Lasserre:Book:CPOPApplications}.
Since it is \NP-hard, e.g., \cite{Laurent:Survey}, one commonly guarantees nonnegativity via a \struc{\textit{certificate of nonnegativity}}, a sufficient, but not necessary condition, which is easier to check than nonnegativity itself.
The standard certificates of nonnegativity are \struc{\textit{sums of squares (SOS)}}, which have a long history dating back to Hilbert; see \cite{Hilbert:Seminal} and see \cite{Reznick:HilbertSurvey} for a historical overview.
SOS certificates can be detected via \struc{\textit{semidefinite programming (SDP)}} and were tremendously successful in recent years; see e.g., \cite{Blekherman:Parrilo:Thomas,Laurent:Survey,Lasserre:Book:CPOPApplications,Lasserre:Book:CPOPIntroduction} for an overview.

The SOS approach has, however, certain limits: In 2006, Blekherman proved that for fixed degree $2d \geq 4$ and $n \to \infty$ almost every nonnegative polynomial is not SOS \cite{Blekherman:ConeComparison}.
Moreover, the SDP, which is equivalent to detecting an $n$-variate degree $d$ SOS, has a matrix-variable of size $\binom{n+d}{d}$.
Thus, the corresponding SDPs are too large to be solved even for moderate size of $n$ and $d$, often even if further preprocessing tools are exploited.

In the recent past, building on work by Reznick \cite{Reznick:AGI}, Iliman and the second author \cite{Iliman:deWolff:Circuits} suggested a certificate based on \struc{\emph{sums of nonnegative circuit polynomials (SONC)}}, a particular class of sparse polynomials, as an alternative to SOS for hard polynomial optimization problems; see \Cref{section:preliminaries} for details.
Generalizing an approach by Ghasemi and Marshall \cite{Ghasemi:Marshall:GP,Ghasemi:Marshall:GP:Semialgebraic}, they showed in \cite{Iliman:deWolff:GP} that SONC certificates can be detected via \struc{\emph{geometric programming (GP)}} if the investigated polynomial has a simplex Newton polytope.
Joint with Dressler and Iliman, the second author provided a first heuristic for computing SONC certificates for polynomials with non-simplex Newton polytope \cite{Dressler:Iliman:deWolff:SONCApproachConstraints}.

While investigated examples suggested significant shorter runtimes for SONC/GP compared to SOS/SDP, the approach had shortcomings so far:
\begin{enumerate}
	\item Due to the non-existence of a software for the SONC approach, examples could only be computed with severe effort.
	\item As a consequence, no large scale comparison of both approaches had been carried out; the existing examples had to be considered as preliminary data.
	\item There existed no deterministic algorithm to compute SONC certificates for polynomials with non-simplex Newton polytope.
\end{enumerate}

In this article, we resolve these shortcomings.
We present an algorithm for computing SONC certificates for \struc{\emph{sparse polynomials}} with arbitrary Newton polytope; see \cref{algorithm:bound}.
The key idea is to generate a covering of the Newton polytope of a given $f$ via solving several linear optimization problems and then considering the distribution of coefficients as part of the GP we have to solve.
We show that this algorithm is not only systematic, but also outperforms the heuristic approach in \cite{Dressler:Iliman:deWolff:SONCApproachConstraints} with respect to the quality of the computed bounds; see \cref{section:describing_the_software} for further details.
Moreover, our \cref{algorithm:bound} runs in \struc{\emph{polynomial time}}, see \cref{theorem:full_algo_poly_time}.

Our main contribution is an experimental large-scale comparison of SONC/GP versus SOS/SDP using these algorithms, which are implemented in our software POEM (Effective Methods in Polynomial Optimization; developed since July 2017) \cite{poem:software}.
The software is available at 
\begin{center}
	\url{https://www3.math.tu-berlin.de/combi/RAAGConOpt/poem.html}.
\end{center}
We let it compete with a broad range of the existing standard SDP solvers, explicitly allowing them to exploit sparsity.
We generated a database of \instances{} test cases covering polynomials with various sorts of Newton polytopes, number of variables in the range of 2 to 40, degree in the range of 6 to 60, and a cardinality of terms in the range of 6 to 500.
Among these are \trivialInstances{} trivial cases, which are sums of monomial squares.

The full database of instances is available at the software homepage; see \cref{subsection:generating_polynomials} for further details on its generation.
The complete database of computations (around 23 GB) is available on inquiry.

We discuss all computational results in detail in \Cref{subsection:evaluation}. In summary, our software executing \cref{algorithm:bound}, combined with the solver ECOS \cite{ecos}, solved most instances within a few seconds.
Even for the largest of our instances it required less than 5 minutes. 
Over all, we obtained a SONC bound in \soncBound{} (98.9 \%) of the 16420 investigated non-trivial cases.
In contrast, we obtained an SOS bound only in 5161 (31.4 \%) out of the 16420 investigated non-trivial cases. 
In 5732 (34.9 \%) cases, where the SOS solvers did not find a bound, they did not terminate due to numerical issues, or ran out of memory. 
The remaining 6028 cases led to SDPs with matrices of sizes beyond $1000 \times 1000$, which we established experimentally as a reasonable threshold for 16 GB of RAM.

In comparison, we achieved a SONC but no SOS bound in 11760 (71.6 \%) of the cases and an SOS but no SONC bound in only 44 (0.0027 \%) of the cases.

Among the $4479$ cases (27.3 \% of the entire nontrivial cases) where we achieved both SOS and SONC bounds, usually the SOS bound is better (by more than $10^{-3}$), namely in 3693 cases (82.5 \%). 
In $782$ of these cases, however, the Newton polytope was a scaled standard simplex, where SOS is at least as good as SONC by construction. 
Moreover, only in 1625 (36.28 \%) of the cases the difference of the bounds was greater than 1.

We visualize the main outcome of our computations in \cref{figure:plot_bar_difference}: 
In the left plot, we depict the number of instances such that the optimality bounds $\struc{\SOSopt}$ and $\struc{\GPopt}$, which we obtain using SOS and the SONC certificates, satisfy: $\GPopt - \SOSopt$ is the interval on the lower axis. 
In the right plot, we present the ratio of the runtime of the fastest SOS method compared to the runtime of our SONC algorithm with ECOS and depict for how many instances the ratio lies in the given interval.

\begin{figure}[t]
	\ifpictures
	\centering
	\input{plot_bar.tex}
	\caption{Number of instances where (left) $\GPopt-\SOSopt$ lies in the interval given on the lower axis; (right) the ratio $\operatorname{SOS}/\operatorname{SONC}$ of the running time lies in the interval given on the lower axis.
	The upper red bar represents instances, where the problem size for SOS lies above our threshold.}
	\label{figure:plot_bar_difference}
	\fi
\end{figure}
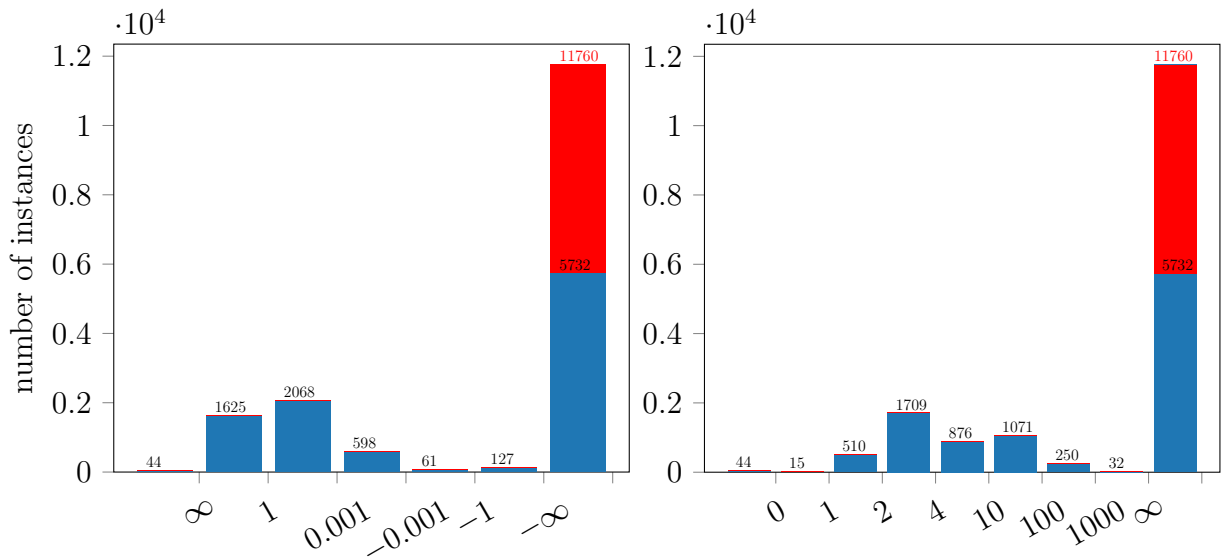

The article is organized as follows: 
In \cref{section:preliminaries} we introduce our notation and previous results about SONCs. 
In \cref{section:describing_the_software} we introduce develop \cref{algorithm:bound} and prove that it runs in polynomial time \cref{theorem:full_algo_poly_time}. 
In \cref{section:running_test_cases} we describe the setup of our experiment and present its outcome. 
Finally, we draw the conclusions from our experiments in \cref{section:conclusion}.

\subsection*{Acknowledgements}

We thank Mareike Dressler, Meike Hatzel, Sadik Iliman and Michael Joswig for their helpful suggestions. 
We thank Benjamin Lorenz for his invaluable technical support.

Both authors of this article are supported by the DFG grant WO 2206/1-1.

\section{Preliminaries}
\label{section:preliminaries}

In this section we introduce our basic notation, sums of squares, sums of nonnegative circuit polynomials, and geometric programs.

\subsection{Representing Sparse Polynomials}

Throughout the paper, we use bold letters for vectors (small) and matrices (capital), e.g., $\struc{\Vector{x}}=(x_1,\ldots,x_n) \in \R^n$.
Let $\struc{\R_{\geq 0}}$ and $\struc{\R_{> 0}}$ be the set of nonnegative and positive real numbers, respectively.
Furthermore, let \struc{$\R[\Vector{x}] = \R[x_1,\ldots,x_n]$} be the ring of real $n$-variate polynomials. 
We denote the set of all $n$-variate polynomials of degree less than or equal to $2d$ by $\struc{\R[\Vector{x}]_{n,2d}}$. 
For $p \in \R[\Vector{x}]$ we denote the \struc{\textit{total degree}} of $p$ by $\struc{\deg(p)}$.

We mostly regard \struc{\emph{sparse polynomials}} $p \in \R[\Vector{x}]$ supported on a finite set $\struc{A} \subset \N^n$; we write $\struc{\support{p}}$ if a clarification is necessary.
Thus, $p$ is of the form $\struc{p(\Vector{x})} = \sum_{\Vector{\alpha} \in A}^{} b_{\Vector{\alpha}}\Vector{x}^{\Vector{\alpha}}$ with $\struc{b_{\Vector{\alpha}}} \in \R\setminus\{0\}$ and $\struc{\Vector{x}^{\Vector{\alpha}}} = x_1^{\alpha_1} \cdots x_n^{\alpha_n}$. 
Unless stated differently, we follow the convention $\struc{t} = \# A$.
The support of $p$ can be expressed as an $n \times t$ matrix, which we denote by $\struc{\Matrix{A}}$
, such that the $j$-th column of $\Matrix{A}$ is $\Vector{\alpha(j)}$.
Hence, $p$ is uniquely described by the pair $(\Matrix{A},\Vector{b})$, written $p = \struc{\polynomial{\Matrix{A}}{\Vector{b}}}$.

We denote by $\struc{\New(p)} = \conv\left(\{\Vector{\alpha} \in \N^n : b_{\Vector{\alpha}} \neq 0\}\right) = \conv\left({\support{p}}\right)$ the \textit{\struc{Newton polytope}} of $p$ with vertices $\struc{\vertices{p}} = \left\{\Vector\alpha \in \support{p} : \Vector{\alpha} \text{ is vertex of } \newton{p}\right\}$.
A lattice point is called \struc{\textit{even}} if it is in $(2\N)^n$ and a term $b_{\Vector{\alpha}}\Vector{x}^{\Vector{\alpha}}$ is called a \struc{\emph{monomial square}} if $b_{\Vector{\alpha}} > 0$ and $\Vector{\alpha}$ even.
We define $\struc{\monoSquares{p}} = \left\{\Vector{\alpha} \in \support{p} : \Vector{\alpha} \in (2\N)^n, b_{\Vector{\alpha}} > 0\right\}$ as the set of monomial squares in the support of $p$.
Moreover, we use the notation $\struc{\nonSquares{p}} = \support{p} \setminus \monoSquares{p}$ for all elements of the support of $p$, which are not a monomial square. We indicate the elements of the support which are in the interior of $\New(p)$ by $\struc{\interior{p}} = \support{p} \setminus \partial \New(p)$.

\subsection{Formulation of the SOS-problem}
\label{subsection:SOS}

Let $p \in \R[\Vector{x}]$. 
In unconstrained polynomial optimization we are interested in computing the infimum of $p$ given by
\begin{align*}
	\begin{aligned}
		\struc{p^*} & = & \inf_{\Vector{x} \in \R^n} p(\Vector{x}) \ = \ \sup\{\gamma \in \R \ : \ p(\Vector{x}) - \gamma \geq 0 \text{ for all } \Vector{x} \in \R^n\}.
	\end{aligned}
\end{align*}
Note that a polynomial is bounded below if and only if its homogenization is nonnegative.
Since deciding nonnegativity is \coNP{}-hard \cite{Laurent:Survey}, deciding, whether $p^*$ is finite, also is \coNP{}-hard.

A polynomial $p$ is a \struc{\textit{sum of squares} (SOS)} if there exist $s_1,\ldots,s_r \in \R[\Vector{x}]$ such that $p = \sum_{j = 1}^r s_j^2$. 
Writing a polynomial as a sum of squares provides a certificate for nonnegativity. 
We denote
\begin{align*}
	\begin{aligned}
		\struc{\SOSopt} & = & \sup\{\gamma \in \R \ : \ p(\Vector{x}) - \gamma \text{ is a sum of squares}\},
	\end{aligned}
\end{align*}
and have $\SOSopt \leq p^*$. Let $2d = \deg(p)$. The bound \SOSopt{} is the optimal value of the following \struc{\textit{semidefinite program (SDP)}}
\begin{mini}
	{}
	{\gamma}
	{\label{problem:prob_sos}\tag{SDP-SOS}}
	{\SOSopt \ = \ }
	\myCons{X_{0,0}}{= b_0 + \gamma}{}
	\myCons{\sum_{\stackrel{\Vector{\beta}\ \leq\ \Vector{\alpha}}{\left|\Vector{\beta}\right|\ \leq\ d}} X_{\Vector{\beta},\Vector{\alpha}-\Vector{\beta}}}{= b_{\Vector{\alpha}}}{1\leq |\Vector{\alpha}|\leq 2d}
	\myCons{\Matrix{X}}{\succeq 0}{}
\end{mini}
where the last conditions means $\Matrix{X}$ is positive semidefinite, and we write $\Vector{\alpha}\leq\Vector{\beta}$ if and only if $\alpha_i\leq\beta_i$ for all $i=1,\ldots,n$. 
For background on SOS and SDPs see e.g., \cite{Blekherman:Parrilo:Thomas,Laurent:Survey}.
Given a feasible starting point and an accuracy $\varepsilon > 0$, an SDP can be solved up to accuracy $\varepsilon$ in time polynomial in the \emph{problem size}.

The size of the problem is, however, the major drawback of this approach: $\Matrix{X}$ is a $\binom{n+d}{d}\times\binom{n+d}{d}$-matrix and we have $\binom{n+2d}{2d}$ constraints.
Thus, if both the number of variables and the degree grow larger, the problem requires an exponentially large amount of RAM, even to state it.
Therefore, in this case, solving the problem becomes practically infeasible, even for moderate sizes of $d$ and $n$.

\subsection{Sums of Nonnegative Circuit Polynomials}

We introduce the fundamental facts of SONC polynomials, which we use in this article. 
SONCs are constructed by \textit{circuit polynomials}; which were first introduced in \cite{Iliman:deWolff:Circuits}:

\begin{definition}
	A \struc{\emph{circuit polynomial}} $p = \poly(\Matrix{A},\Vector{b}) \in \R[\Vector{x}]$ is of the form
	\begin{eqnarray}
		\struc{p(\Vector{x})} & = & \sum_{j=0}^r b_{\Vector{\alpha}(j)} \Vector{x}^{\Vector{\alpha}(j)} + b_{\Vector{\beta}} \Vector{x}^{\Vector{\beta}}, \label{Equ:CircuitPolynomial}
	\end{eqnarray}
	with $0 \leq \struc{r} \leq n$, coefficients $\struc{b_{\Vector{\alpha}(j)}} \in \R_{> 0}$, $\struc{b_{\Vector{\beta}}} \in \R$, exponents $\struc{\Vector{\alpha}(j)} \in (2\Z)^n$, $\struc{\Vector{\beta}} \in \Z^n$, such that the following condition holds:
	There exist unique, positive \struc{\emph{barycentric coordinates} $\lambda_j$} relative to the $\Vector{\alpha}(j)$ with $j=0,\ldots,r$ satisfying
        \begin{align}
            \label{equ:BarycentricCoordinates}
            & & \Vector{\beta} \ = \ \sum_{j=0}^r \lambda_j \Vector{\alpha}(j) \ \text{ with } \ \lambda_j \ > \ 0 \ \text{ and } \ \sum_{j=0}^r \lambda_j \ = \ 1.
        \end{align}
    For every circuit polynomial $p$ we define the corresponding \struc{\textit{circuit number}} as
	\begin{align}
		\label{equ:DefCircuitNumber}
		\struc{\Theta_p} \ = \ \prod_{j = 0}^r \left(\frac{b_{\Vector{\alpha}(j)}}{\lambda_j}\right)^{\lambda_j}. 
	\end{align}
	\label{Def:CircuitPolynomial}
\end{definition}
	Condition \eqref{equ:BarycentricCoordinates} implies that $\Matrix{A}(p)$ forms a minimal affine dependent set. 
	Those sets are called \struc{\textit{circuits}}, see e.g., \cite{Oxley:MatroidTheory}.
	More specifically, Condition \eqref{equ:BarycentricCoordinates} yields that $\New(p)$ is a simplex with even vertices $\Vector{\alpha}(0), \Vector{\alpha}(1),\ldots,\Vector{\alpha}(r)$ and that the exponent $\Vector{\beta}$ is in the strict interior of $\New(p)$ if $\dim(\New(p)) \geq 1$. 
	Therefore, we call the terms $p_{\Vector{\alpha}(0)} \Vector{x}^{\Vector{\alpha}(0)},\ldots,p_{\Vector{\alpha}(r)} \Vector{x}^{\Vector{\alpha}(r)}$ the \struc{\emph{outer terms}} and $p_{\Vector{\beta}} \Vector{x}^{\Vector{\beta}}$ the \struc{\emph{inner term}} of $p$.

Circuit polynomials are proper building blocks for nonnegativity certificates since the circuit number alone determines whether they are nonnegative.

\begin{theorem}[\cite{Iliman:deWolff:Circuits}, Theorem 3.8]
	Let $p$ be a circuit polynomial of the form \eqref{Equ:CircuitPolynomial}. Then $p$ is nonnegative if and only if:
	\begin{enumerate}
		\item $p$ is a sum of monomial squares, or
		\item the coefficient $b_{\Vector{\beta}}$ of the inner term of $p$ satisfies $|b_{\Vector{\beta}}| \leq \Theta_p$.
	\end{enumerate}
	\label{thm:CircuitPolynomialNonnegativity}
\end{theorem}

Note that $\Theta_p$ can be computed by solving a system of linear equations if it is known that $\Vector{\beta}$ is in the relative interior of $\New(p)$ and by solving an LP otherwise.

\begin{definition}
	We define for every $n,d \in \N$ the set of \struc{\emph{sums of nonnegative circuit polynomials} (SONC)} in $n$ variables of degree $2d$ as
	\begin{align*}
		\struc{C_{n,2d}} \ = \ \left\{f \in \R[\Vector{x}]_{n,2d} \ :\ f = \sum_{\rm finite}  p_i, \quad p_i \text{ is a nonnegative circuit polynomial} \right\}.
	\end{align*}
	\label{Def:SONC}
\end{definition}

We denote by SONC both the set of SONC polynomials and the property of a polynomial to be a sum of nonnegative circuit polynomials.

In what follows let $\struc{P_{n,2d}}$ be the cone of nonnegative $n$-variate polynomials of degree at most $2d$ and $\struc{\Sigma_{n,2d}}$ be the corresponding cone of sums of squares respectively. 
An important observation is, that SONC polynomials form a convex cone independent of the SOS cone:

\begin{theorem}[\cite{Iliman:deWolff:Circuits}, Proposition 7.2]
	$C_{n,2d}$ is a convex cone satisfying:
	\begin{enumerate}
		\item $C_{n,2d} \subseteq P_{n,2d}$ for all $n,d \in \N$,
		\item $C_{n,2d} \subseteq \Sigma_{n,2d}$ if and only if $(n,2d)\in\{(1,2d),(n,2),(2,4)\}$,
		\item $\Sigma_{n,2d} \not\subseteq C_{n,2d}$ for all $(n,2d)$ with $2d \geq 6$.
	\end{enumerate}
	\label{Thm:ConeContainment}
\end{theorem}

For further details about the SONC cone see \cite{deWolff:Circuits:OWR,Iliman:deWolff:Circuits, Dressler:Iliman:deWolff:Positivstellensatz}.

\subsection{Previous Methods for the Computation of SONC Decompositions}

Geometric programming was introduced in \cite{Duffin:Peterson:Zener:Book}. 
It is equivalent to a convex optimization problem. Applications include problems in circuit design problems, nonlinear network flow, and optimal control; for an overview see \cite{Boyd:Kim:Vandenberghe:TutorialOnGP, Boyd:Vandenberghe:ConvexOptimization}

\begin{definition}
	A function $p : \R_{>0}^n\to \R$ of the form $\struc{p(\Vector{z})} = p(z_1,\ldots,z_n) = cz_1^{\alpha_1}\cdots z_n^{\alpha_n}$ with $c > 0$ and $\alpha_i \in \R$ is called a \struc{\emph{monomial (function)}}. A sum $\struc{\sum_{i=0}^k c_iz_1^{\alpha_{1}(i)}\cdots z_n^{\alpha_{n}(i)}}$ of monomials with $c_i > 0$ is called a \struc{\emph{posynomial (function)}}.

	A \struc{\emph{geometric program} (GP)} has the following form:
	\begin{mini}
		{\Vector{z}\in\R_{>0}^n}
		{p_0(\Vector{z})}
		{\label{prob:mine}}
		{}
		\myCons{p_i(\Vector{z})}{\leq 1}{i=1,\ldots,N}
		\myCons{q_j(\Vector{z})}{= 1}{j=1,\ldots,M}
	\end{mini}
	where $p_0,\dots,p_m$ are posynomials and $q_1,\dots,q_l$ are monomial functions.
	\label{definition:GP}
\end{definition}

Geometric programs are convex and can hence be solved with interior point methods. In \cite{Nesterov:Nemirovskii}, the authors prove worst-case polynomial time complexity of this method; see also \cite[Page 118]{Boyd:Kim:Vandenberghe:TutorialOnGP}.

In \cite{Iliman:deWolff:GP} Iliman and the second author used GP and SONC to compute lower bounds for polynomials if their Newton polytope is a simplex.
For the case of an non-simplex Newton polytope, Dressler, Iliman and the second author propose the following heuristic approach to obtain lower bound for polynomials via SONC in \cite[page 20]{Dressler:Iliman:deWolff:SONCApproachConstraints}.
\begin{enumerate}
	\item Choose a triangulation $T_1,\ldots,T_\ell$ of the points corresponding to monomial squares.
	\item Choose weights $f_{\Vector{\beta},i}$ for all $(\Vector{\beta},i)\in \support{p}\times\{1,\ldots,\ell\}$ with $\Vector{\beta}\in T_i$, satisfying $\sum_{i=1}^\ell f_{\Vector{\beta},i}=f_{\Vector{\beta}}$.
	\item Define polynomials 
		\begin{align*}
			g_i \ = \ \sum_{\beta \in T_i\cap\support{p}} f_{\Vector{\beta},i} \Vector{x}^{\Vector{\beta}} \text{ for all $1\leq i\leq \ell$}
		\end{align*}
\end{enumerate}
Then each $\newton{g_i}$ is a simplex, so we can apply the previous approach.
However, they do not provide methods to obtain either a good triangulation or a good distribution of the weights, see \cref{example:DIdW16_5.5}.

\section{A Polynomial Time Algorithm for Computing SONC Certificates}
\label{section:describing_the_software}

Let $p=\polynomial{\Matrix{A}}{\Vector{b}}$ be some polynomial with $\Matrix{A}\in\N^{n\times t}$ and $\Vector{b}\in\R^t$.
In this section we describe an algorithm to find a lower bound for $p$.
We begin with recalling the special case where the Newton polytope forms a simplex and every point, which is not a vertex, lies in the interior.
Afterwards, we describe two ways for handling arbitrary polynomials.
The first approach splits the problem into several instances of the simplex case.
In the second approach we have to solve a single larger optimization problem, which is similar to the problem from the simplex case.
The first approach can be parallelized easily, while the second yields better results.
Our initial hypothesis was that the first approach would run significantly faster, which turned out not to be the case.
Regarding a comparison of run times, see \cref{table:split_strategy_time}.

As a first relaxation, every coefficient of a monomial non-square gets a negative sign.
Their exponents form the \struc{\emph{negative points}}, while the exponents of the monomial squares form the \struc{\emph{positive points}}.
This corresponds to the worst case where every possibly negative term is negative for the same argument.
Additionally, we can thus focus on the positive orthant.
For each negative point we select positive points, which form a simplex, that contains the negative point in its interior.
Next, we identify all other negative points, which lie in this simplex.
We continue this procedure until each negative point is covered by at least one simplex.
Then we fix a distribution of the negative coefficients and compute an optimal distribution of the positive coefficients via GP, such that each simplex is split into nonnegative circuit polynomials.
The adjustment we have to make for the constant term thus yields a lower bound for the polynomial.

\subsection{Computing the Newton polytope}
\label{subsection:convex_hull}

Recall that a point in a set $A \subset \R^n$ is \struc{\textit{extremal}} if and only if it cannot be written as a convex combination of the other points in $A$.
We need an efficient algorithm for deciding whether the Newton polytope corresponding to a given polynomial $p$ is a simplex.
This means, we have to find the extremal points of the set of exponents.
So, for every $\Vector{v} \in A(p)$ we ask for the feasibility of the following LP:
\begin{align}
	\begin{aligned}
		\label{LP:extremal}
		\sum_{\Vector u\in\support{p}\setminus\{\Vector v\}} \lambda_{\Vector{u}} \cdot \Vector{u} &= \Vector{v}\\
		\sum_{\Vector u\in\support{p}\setminus\{\Vector v\}} \lambda_{\Vector{u}} &= 1\\
		\lambda_{\Vector{u}} &\geq 0 \qquad\text{for all } \Vector{u}\in\support{p}\setminus\{\Vector{v}\}
	\end{aligned}
	\tag{LP-extremal}
\end{align}
which can be decided in polynomial time \cite{Khachiyan:LPPolytime}.
Solving this LP for every $\Vector{v}\in\support{p}$, computes $\vertices{p}$ in polynomial time.
Furthermore, the solutions of these LPs yield convex combinations of the interior points with respect to the vertices.
For later we need the following result, which is slightly stronger.
The lemma is folklore.
We, however, provide a short proof for convenience of the reader.

\begin{lemma}
	\label{lem:cover_reduction}
	For every non-extremal point $\Vector{v}\in \support{p}\setminus\vertices{p}$, we can efficiently compute affinely independent $\Vector{v}_0,...,\Vector{v}_m \in \vertices{p}$ with $m\leq n$ such that and $\Vector{v}\in\conv\left(\{\Vector{v}_0,...,\Vector{v}_m\}\right)$.
\end{lemma}
\begin{proof}
	Solve \cref{LP:extremal} to obtain a convex combination $\Vector{v} = \sum_{i=0}^m\lambda_i \Vector{v}_i$.
	Assume $m>n$, and let $\mu$ be a non-trivial element in
	\begin{align}
		\label{Proof:LP:Matrix}
		\operatorname{ker}
		\begin{pmatrix}
			1 & \cdots & 1 \\
			\Vector{v}_0 & \cdots & \Vector{v}_m
		\end{pmatrix}
	\end{align}
	Assign $\eta = \min\left\{\frac{\lambda_i}{\mu_i}: i\leq m, \mu_i>0\right\}$ and choose coefficients $\lambda_i' = \lambda_i - \eta\mu_i$.
	Then $\Vector{v} = \sum_{i=0}^m (\lambda_i'-\eta\mu_i)\Vector{v}_i$ is a new convex combination but one of the coefficients vanished, so we obtain a shorter convex combination of $\Vector{v}$.
	We delete the corresponding column in the matrix \eqref{Proof:LP:Matrix} and iterate the process until we obtain a matrix with trivial kernel.
	Once this is the case the vectors $\Vector{v}_i-\Vector{v}_0$ are affinely independent.
	Note that this can only occur if $m\leq n$,
\end{proof}

\subsection{The SONC-Problem for Simplex Newton polytope}
\label{sec:simplex}

In this part, we assume $\support{p} = \vertices{p} \cup \interior{p}$ and that $\newton{p}$ is a simplex of dimension $\struc{h}\leq n$.
We recall how to compute SONC certificates using GPs in this case. This was first shown by Iliman and the second author in \cite{Iliman:deWolff:GP} generalizing an idea of Ghasemi and Marshall \cite{ghasemi_marshall_2012}, where the latter regard the problem only for the case that $\newton{p}$ is the standard simplex using theoretical results by Fidalgo and Kovacec \cite{Fidalgo:Kovacec}.

Let $\Matrix{\lambda}$ be the barycentric coordinates of $\interior{p}$ with respect to $\vertices{p}$, i.e. $\vertices{p}\cdot \Matrix{\lambda}= \interior{p}$ in matrix notation.
Note that $\Matrix{\lambda}$ is unique, since $\vertices{p}$ forms a simplex.
This means we write each interior point as a convex combination of the vertices of $\New(p)$.
In particular, we have $0< \lambda_{i,j}< 1$ for all $i,j$.
We define \struc{\GPopt{}} as the solution of the following GP.
\begin{mini}
	{}
	{\sum_{j=0}^{t-h-1} X_{0,j}}
	{\label{problem:sonc}\tag{SONC-simplex}}
	{\GPopt \ = \ }
	\myCons{\sum_{j=0}^{t-h-1}X_{i,j}}{\leq b_i}{1\leq i<h}
	\myCons{\prod_{i=0}^{h-1}\left(\frac{X_{i,j}}{\lambda_{i,j}}\right)^{\lambda_{i,j}}}{=-b_{h+j}}{j<t-h}
\end{mini}
Following \cref{thm:CircuitPolynomialNonnegativity}, the second set of constraints would have a ``$\geq$'', but for the optimum we always have equality.
Hence, we can use the above form, so we have a GP.
%\begin{align}
%	\begin{aligned}
%		\label{problem:sonc}
%		\GPopt =\,&\min & \sum_{j=0}^{t-h-1} {X_{0,j}}&\\
%		&\st & \sum_{j=0}^{t-h-1} X_{i,j} &\leq b_i && \text{for }1\leq i<h\\%i=1,\ldots,h-1\\
%		&& \prod_{i=0}^{h-1} \left(\frac{X_{i,j}}{\lambda_{i,j}}\right)^{\lambda_{i,j}} & = -b_{h+j} && \text{for }j<t-h%j=0,\ldots,t-h-1
%	\end{aligned}
%	\tag{SONC-simplex}
%\end{align}

Assume that \GPopt{} is attained at $\Matrix{\GPsol}$ and let $\gamma=\GPopt-b_0$.
Then $\Matrix{\GPsol}$ is a nonnegativity certificate for $p+\gamma$ in the following way. 
For $0\leq j<t-h$ let 
\begin{align*}
	f_j \ = \ \sum_{i=0}^{h-1} \GPsol_{i,j} \cdot \Vector{x}^{\Vector{\alpha}(i)} + b_{h+j}\Vector{x}^{\Vector{\alpha}(h+j)}.
\end{align*}
Then each $f_j$ is nonnegative due to the second constraint and \cref{thm:CircuitPolynomialNonnegativity}.
Therefore,
\begin{align*}
	p+\gamma \ = \ \sum_{j=0}^{t-h-1} f_i \geq 0
\end{align*}
is a sum of nonnegative circuit polynomials. Thus, $-\gamma$ is a lower bound for $p$.

\medskip

Regarding the feasibility of the problem \cref{problem:sonc} we observe:

\begin{proposition}
	\label{lem:GP_feasible}
	The problem \cref{problem:sonc} is always feasible.
\end{proposition}
\begin{proof}
	Assume we have some assignment for $\Matrix{X}$.
	Then we get a feasible solution $\Matrix{X'}$ via
	\begin{align*}
	% 		\label{eq:feasibility}
		X'_{i,j} &= \frac{X_{i,j}}{1 - b_i +\sum_{j=0}^h X_{i,j}} &&\text{for $i<h$}\\
		X'_{0,j} &= \left(-b_{h+j}\cdot \prod_{i=1}^{h} \left(\frac{X'_{i,j}}{\lambda_i}\right)^{\lambda_i}\right)^{\frac{1}{\lambda_0}} &&
		\text{for $j<t-h$}
		\qedhere
	\end{align*}
\end{proof}

\subsection{The SONC-problem for Arbitrary Polynomials}
\label{subsection:arbitrary_newton}

In order to compute a lower bound for an arbitrary polynomial $p$, we cover $\New(p)$ by simplices and solve a variation of the problem previously discussed.

For each $\Vector{\alpha}\in\nonSquares{p}$ we have to find a simplex spanned by elements of $\monoSquares{p}$ such that $\Vector{\alpha}$ is contained in the interior of this simplex.
In this way we cover all elements of $\nonSquares{p}$ with simplices. Note that this covering is \emph{not} a triangulation in general.
Since all vertices of the Newton polytope are monomial squares by assumption (otherwise we have $p^* = -\infty$ trivially) such a cover exists.

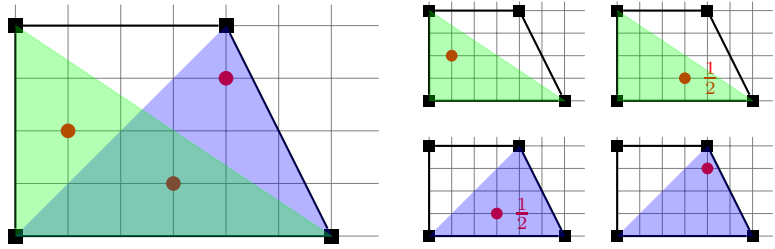
\begin{figure}[t]
	\ifpictures
	\newcommand{\halfweight}{\footnotesize{\textcolor{red}{$\frac{1}{2}$}}}
		\begin{tikzpicture}[square/.style={regular polygon,regular polygon sides=4}]
			\path[use as bounding box] (-4,2) rectangle (6,-2);
			\node at (-2,0) {
				\begin{tikzpicture}[scale = 0.7]
					\path[use as bounding box] (0,0) rectangle (7.5,5);
					\draw[help lines] (0,0) grid (6.9,4.4);
					\node[fill, square, scale = 0.5] (a0) at (0,0) {};
					\node[fill, square, scale = 0.5] (a1) at (6,0) {};
					\node[fill, square, scale = 0.5] (a2) at (4,4) {};
					\node[fill, square, scale = 0.5] (a3) at (0,4) {};
					\draw[thick] (a0) -- (a1) -- (a2) -- (a3) -- (a0);
					\node[fill, circle, scale = 0.5, red] (c1) at (3,1) {};
					\node[fill, circle, scale = 0.5, red] (c2) at (4,3) {};
					\node[fill, circle, scale = 0.5, red] (c3) at (1,2) {};

					\fill[color = blue, opacity=0.3, draw = blue]   (0,0) -- (6,0) -- (4,4);
					\fill[color = green, opacity=0.3, draw = green] (0,0) -- (6,0) -- (0,4);
				\end{tikzpicture}
			};
			\node at (2,0.8) {
				\begin{tikzpicture}[scale = 0.3]
					\path[use as bounding box] (0,0) rectangle (7.5,5);
					\draw[help lines] (0,0) grid (6.9,4.4);
					\node[fill, square, scale = 0.4] (a0) at (0,0) {};
					\node[fill, square, scale = 0.4] (a1) at (6,0) {};
					\node[fill, square, scale = 0.4] (a2) at (4,4) {};
					\node[fill, square, scale = 0.4] (a3) at (0,4) {};
					\draw[thick] (a0) -- (a1) -- (a2) -- (a3) -- (a0);
					\node[fill, circle, scale = 0.4, red] (c3) at (1,2) {};

					\fill[color = green, opacity=0.3, draw = green] (0,0) -- (6,0) -- (0,4);
				\end{tikzpicture}
			};
			\node at (4.5,0.8) {
				\begin{tikzpicture}[scale = 0.3]
					\path[use as bounding box] (0,0) rectangle (7.5,5);
					\draw[help lines] (0,0) grid (6.9,4.4);
					\node[fill, square, scale = 0.4] (a0) at (0,0) {};
					\node[fill, square, scale = 0.4] (a1) at (6,0) {};
					\node[fill, square, scale = 0.4] (a2) at (4,4) {};
					\node[fill, square, scale = 0.4] (a3) at (0,4) {};
					\draw[thick] (a0) -- (a1) -- (a2) -- (a3) -- (a0);
					\node[fill, circle, scale = 0.4, red, label=right:\halfweight] (c1) at (3,1) {};

					\fill[color = green, opacity=0.3, draw = green] (0,0) -- (6,0) -- (0,4);
				\end{tikzpicture}
			};
			\node at (2,-1) {
				\begin{tikzpicture}[scale = 0.3]
					\path[use as bounding box] (0,0) rectangle (7.5,5);
					\draw[help lines] (0,0) grid (6.9,4.4);
					\node[fill, square, scale = 0.4] (a0) at (0,0) {};
					\node[fill, square, scale = 0.4] (a1) at (6,0) {};
					\node[fill, square, scale = 0.4] (a2) at (4,4) {};
					\node[fill, square, scale = 0.4] (a3) at (0,4) {};
					\draw[thick] (a0) -- (a1) -- (a2) -- (a3) -- (a0);
					\node[fill, circle, scale = 0.4, red, label=right:\halfweight] (c1) at (3,1) {};

					\fill[color = blue, opacity=0.3, draw = blue]   (0,0) -- (6,0) -- (4,4);
				\end{tikzpicture}
			};
			\node at (4.5,-1) {
				\begin{tikzpicture}[scale = 0.3]
					\path[use as bounding box] (0,0) rectangle (7.5,5);
					\draw[help lines] (0,0) grid (6.9,4.4);
					\node[fill, square, scale = 0.4] (a0) at (0,0) {};
					\node[fill, square, scale = 0.4] (a1) at (6,0) {};
					\node[fill, square, scale = 0.4] (a2) at (4,4) {};
					\node[fill, square, scale = 0.4] (a3) at (0,4) {};
					\draw[thick] (a0) -- (a1) -- (a2) -- (a3) -- (a0);
					\node[fill, circle, scale = 0.4, red] (c2) at (4,3) {};

					\fill[color = blue, opacity=0.3, draw = blue]   (0,0) -- (6,0) -- (4,4);
				\end{tikzpicture}
			};
		\end{tikzpicture}
	\caption{Covering negative terms (red dots) with simplices (colored triangles) whose vertices are monomial squares (black squares).
 	If a vertex occurs in more than one simplex, then its coefficient is split evenly.}
	\label{figure:decomposition}
	\fi
\end{figure}

To compute the cover, we define the following auxiliary program
\begin{maxi}
	{}
	{\lambda_{\Vector{0}}}
	{\label{problem:cover}\tag{LPHull}}
	{\operatorname{LPHull}(\Vector{u}, \monoSquares{p})\ =\ }
	\myCons{\sum_{\Vector{v} \in \monoSquares{p}} \lambda_{\Vector{v}}\cdot \Vector{v}}{= \Vector{u}}{}
	\myCons{\sum_{\Vector{v} \in \monoSquares{p}} \lambda_{\Vector{v}}}{= 1}{}
	\myCons{\lambda_{\Vector{v}}}{\geq 0}{\Vector{v}\in \monoSquares{p}}
\end{maxi}

\begin{algorithm}
	The cover can be computed in polynomial time using the following algorithm.
	The idea is illustrated in \cref{figure:decomposition}.\\
	\INPUT{$\monoSquares{p}, \nonSquares{p}$: sets of points in $\R^n$}\\
	\OUTPUT{$(A^1,\ldots,A^\coverLength)$: sequence of sets of points, such that each $A^i\subseteq\monoSquares{p}$ forms the vertices of a simplex; and each $\Vector{u}\in\nonSquares{p}$ appears in the strict interior of at least one $A^i$.}
	\begin{algorithmic}[1]
		\Function{Cover}{\monoSquares{p}, \nonSquares{p}}
		\State $U := \nonSquares{p}$ \COMMENT{uncovered points}
		\State $k := 0$
		\While{$U \neq \emptyset$}
			\State Choose $\Vector{u}\in U$
			\State $k := k + 1$
			\State $\Vector{\lambda}^* := \operatorname{LPHull}(\Vector{u},\monoSquares{p})$ \COMMENT{solution vector via simplex method}
			\State $S := \{\Vector{v} \in \monoSquares{p}: \Vector{\lambda}^*_{\Vector{v}} > 0\}$ \COMMENT{This is a simplex.}
			\State $A^k := S \cup (\nonSquares{p} \cap \interior{\conv(S)})$
			\State $U := U \setminus A^k$
		\EndWhile
		\State \Return $(A^1, \ldots, A^k)$
		\EndFunction
	\end{algorithmic}
	\label{algorithm:cover}
\end{algorithm}
\begin{proof}
	For the running time, we observe:
	\begin{itemize}
		\item Solving \cref{problem:cover} with the simplex method ensures that $\Vector{\lambda}^*$ has at most $n+1$ non-zero entries, so $S$ is a simplex.
			In theory, however, this algorithm may need exponential time.
			Alternatively, we can solve the \cref{problem:cover} using some method with ensured polynomial running time.
			Then our solution $\Vector{\lambda}^*$ can contain more than $n+1$ non-zero entries and the resulting set $S'=\{\Vector{v} \in \monoSquares{p}: \Vector{\lambda}^*_{\Vector{v}} > 0\}$ is affinely dependent.
			In this case, we apply \cref{lem:cover_reduction}, which reduces $S'$ to an affinely independent set $S$ in polynomial time.
		\item To check whether some point lies in $\interior{\conv(S)}$, we have to solve a linear equation system.
			With QR-factorisation, we can thus compute $A^k$ in $\mathcal O\left(n^2\cdot (n+|\nonSquares{p}|)\right)$ steps.
		\item
			In each iteration of the while-loop we have $\Vector{u} \in \interior{\conv(S)}$, so at the end $\Vector{u} \in A^k$ holds.
			Therefore, at least one element is removed from $U$ in every iteration of the loop.
			Hence, the while-loop has at most $|\nonSquares{p}|$ iterations.
	\end{itemize}
	Each $A^k$ is a simplex with some interior points.
	Their vertices correspond to monomial squares of $p$, and each non-square is contained in at least one simplex.
	Thus, we have a polynomial time algorithm to cover all non-squares by simplices.
\end{proof}

\cref{algorithm:cover} yields a list of sets $A^1,\ldots,A^\coverLength$, where each $A^k$ represents exponents of a polynomial whose Newton polytope is a simplex.
Next we investigate how to split the coefficients in order to obtain nonnegative circuit polynomials. 
We describe two possible strategies.

\subsubsection{Even splitting of the coefficients}
\label{subsubsection:even_split}
For each $i$ where $b_i x^{\Vector{\alpha}(i)}$ is a monomial square, we count how many simplices contain $\Vector{\alpha}(i)$ and evenly distribute $b_i$.
So we define
\begin{align}
	b_i^k & \ = \ \frac{b_i}{\#\{j:\Vector{\alpha}(i) \in A^j\}} &
	p^k & \ = \polynomial{A^k}{\Vector{b}^k}
	\label{eq:split}
\end{align}
restricting $\Vector{b}^k$ to the exponents occurring in $A^k$.
For each $p^k$ we solve the corresponding GP \cref{problem:sonc} from \Cref{sec:simplex}.
If a polynomial $p^k$ does not contain a constant term, then we adapt the GP \cref{problem:sonc} and check for feasibility of
\begin{align}
	\begin{aligned}
		\label{problem:sonc_feasible}
		\sum_{j=0}^{t-h-1} X_{i,j}^k & \ \leq \ b_i^k && \text{for $i < h$}\\
		\prod_{i=0}^{h-1} \left(\frac{X_{i,j}^k}{\lambda_{i,j}^k}\right)^{\lambda_{i,j}} & \ = \ -b_{h+j}^k && \text{for $j<t-h$}
	\end{aligned}
	\tag{GP-Feasibility}
\end{align}
with optimum $\GPopt^k = 0$ if it is feasible and $\GPopt^k = \infty$ otherwise.
So with $\GPopt = \sum_{k=1}^\coverLength \GPopt^k$ we get $-\GPopt$ as a lower bound for the polynomial.

\subsubsection{Variable splitting of the coefficients}
\label{subsubsection:variable_split}
For $\Vector{\alpha}(i)\in\nonSquares{p}$ we compute $b_i^k$ as in \cref{eq:split}.
Then we solve the following GP
\begin{mini}
	{}
	{\sum_{k=1}^\coverLength \sum_{j=0}^{t-h-1} {X_{0,j}^k}}
	{\label{problem:sonc_var_split}\tag{GP-SONC}}
	{}
	\myCons{\sum_{k=1}^\coverLength \sum_{j=0}^{t-h-1} X_{i,j}^k}{ \ \leq \ b_i}{1\leq i< h,}
	\myCons{\prod_{i=0}^{h-1} \left(\frac{X_{i,j}^k}{\lambda_{i,j}^k}\right)^{\lambda_{i,j}^k}}{\ = \ -b_{h+j}^k}{j<t-h, 1\leq k \leq\coverLength.}
	%%% USING THE ARRAY ENVIRONMENT HERE CAUSES AN EMERGENCY STOP IN THE TEX LIVE IN ARXIV
% 	{\begin{array}{l}j<t-h,\\ 1\leq k \leq\coverLength.\end{array}}
\end{mini}

The meaning of the variable is, that we use $X_{i,j}^k$ weight from $b_i$ to balance the negative weight $b_{h+j}^k$ in the simplex given by $A^k$.
So we fix a cover and a distribution of the negative weights and then search for the optimal distribution of the positive weights.

Fixing $b_{h+j}^k$ is necessary in this approach, since if we additionally searched for an optimal distribution of the negative coefficients, the resulting problem would no longer be a GP.

\newcommand{\opt}{\variableStyle{opt}}

\begin{algorithm}
	\label{algorithm:bound}
	The full algorithm to compute a lower bound for a polynomial via SONC is as follows.\\
	\INPUT{$p$: polynomial}\\
	\OUTPUT{$\opt$: float, such that $p+\opt$ is a SONC}
	\begin{algorithmic}[1]
		\Function{Bound}{$p$}
			\State compute $\newton{p}$ and $\vertices{p}$
			\If{$\vertices{p} \cap \nonSquares{p} \neq \emptyset$}
				\State \Return 'unbounded'
			\EndIf
			\If{$\newton{p}$ is simplex}
				\State $X := $ solution of \cref{problem:sonc}
			\Else
				\State $(A^1,\ldots,A^\coverLength) = \operatorname{Cover}(\monoSquares{p}, \nonSquares{p})$
				\State $b_i^k := \frac{b_i}{\#\{j:\Vector{\alpha}(i) \in A^j\}}$
				\State \COMMENT{according to split choice}
				\State solve \cref{problem:sonc_var_split} or several of \cref{problem:sonc} 
			\EndIf
			\State $\opt := \sum_i X_{0,i} - b_{\Vector{0}}$
			\State \Return $\opt$
		\EndFunction
	\end{algorithmic}
\end{algorithm}

\begin{theorem}
	\Cref{algorithm:bound} runs in polynomial time.
	\label{theorem:full_algo_poly_time}
\end{theorem}
\begin{proof}
	As shown in \cref{algorithm:cover}, we have $\coverLength\leq t$.
	Hence, the size of the GP is polynomially bounded in the input size.
	Furthermore, computing the convex hull and solving a GP can be done in polynomial time, see \Cref{subsection:convex_hull} and \cite{Nesterov:Nemirovskii}.
	Therefore, this algorithm computes a lower bound for a polynomial and requires time polynomial in the input size.
\end{proof}

Additionally, the algorithm shows that SONC is in particular well-suited for polynomials of high degree.
\begin{theorem}
	The running time for SONC is independent of the degree in the sense that the following numbers are independent of the degree:
	\begin{itemize}
		\item the sizes of the LPs to compute the Newton polytope,
		\item the sizes of the LPs in \cref{algorithm:cover},
		\item the number of arithmetic operations, to compute the barycentric coordinates,
		\item the problem size of each type of GP we have.
	\end{itemize}
	\label{theorem:degree_independent}
\end{theorem}
\begin{proof}
	The size of both \cref{LP:extremal} and\cref{problem:cover} is in $\mathcal O(nt)$.
	We have to compute $(t-h)\cdot\coverLength\leq t^2$ tuples of barycentric coordinates, each of size $n+1$.
	The largest GP \cref{problem:sonc_var_split} has size $\mathcal O(h(t-h)\coverLength)\subseteq\mathcal O(t^3)$.
	So none of these depends on $d$.
\end{proof}

\subsection{Limits of the Algorithm}

Here, we address cases, where \cref{algorithm:bound} may fail to find a bound, although the polynomial is bounded.
The key concept is a \struc{\textit{degenerate term}}, which is defined as a term $b_{\Vector{\alpha}}\Vector{x}^{\Vector{\alpha}}\in\nonSquares{p}$, such that $\Vector{\alpha}\in\partial\newton{p}$, and in the covering there is a simplex $A^k$ with $\Vector{\alpha}\in A^k$ but $\Vector{0}\notin A^k$.
Then $\Vector{\alpha}$ is called a \struc{\textit{degenerate point}}.

The main issue is, that adding weight to the constant term does not influence non\-ne\-ga\-ti\-vi\-ty of the circuit polynomial with support $A^k$.
Hence, we can use at most the given weights of the other positive points in $A^k$.
If some degenerate term forces us to use the entire weight for every positive term, then there might not be enough weight left to even out other negative terms.

\begin{example}
	\label{example:degenerate}
	As example take
	\begin{align*}
		p = x^2-2xy + y^2 -2x-2y+1 = (x+y-1)^2\geq 0
	\end{align*}
	The only covering is $\left(\{(0,2),(1,1),(2,0)\}, \{(2,0),(1,0),(0,0)\},\{(0,2),(0,1),(0,0)\}\right)$.
	To make the first circuit nonnegative, we have to use the entire coefficients of $x^2$ and $y^2$, to get the nonnegative circuit polynomial $x^2-2xy+y^2$.
	Then however, we have to decompose $-2x-2y+1$ as a SONC, which is not possible, since the expression is unbounded.
\end{example}

\section{Experimental Results}
\label{section:running_test_cases}

We discuss the actual physical running time of the algorithms described above.
First, we describe the setup of our experiment and explain how our random instances were created.
Afterwards, we  discuss a few selected examples, which exhibit well the differences of the methods. 
In the end, we present how the program behaved on a large set of examples.

\subsection{Experimental Setup}
\label{subsection:experimental:setup}

We give an overview about the experimental setup.

\begin{description}
	\item[Software] The entire experiment was steered by our \textsc{Python} based software \textsc{POEM} (Effective Methods in Polynomial Optimization), \cite{poem:software}, which we develop since July 2017. 
	\textsc{POEM} is open source, under GNU public license, and freely available at:
		\begin{center}
			\url{https://www3.math.tu-berlin.de/combi/RAAGConOpt/poem.html}
		\end{center}
		For our experiment, \textsc{POEM} calls a range of further software and solvers for computing both SONC and SOS certificates.
		For SONC, on the one hand, we use \textsc{CVXPY} \cite{cvxpy} with the solver \textsc{ECOS} \cite{ecos}.
		On the other hand, we solve the problems with \textsc{Matlab}, using \textsc{CVX} \cite{cvx_article}, calling the solvers \textsc{SDPT3} \cite{sdpt3} and \textsc{SeDuMi} \cite{sedumi}.
		For SOS, we use the established packages \textsc{YALMIP} \cite{yalmip}, \textsc{Gloptipoly} \cite{gloptipoly} and \textsc{SOSTOOLS} \cite{sostools}, the latter both with and without the ``sparse'' option.
		Furthermore, we provide our own implementation to construct the SDP for SOS.
		For \textsc{Matlab} we use \textsc{CVX} with \textsc{SDPT3} and \textsc{SeDuMi}, and for Python we use \textsc{CVXPY} with \textsc{CVXOPT} \cite{Andersen:Dahl:Vandenberghe:CVXOPT}.
	\item[Investigated Data] The experiment was carried out on a database containing \instances{} polynomials with a wide range of variables, terms, degrees, and Newton polytopes. 
		We created the database randomly using \textsc{POEM}. 
		Further details can be found in \cref{subsection:generating_polynomials}; the full database of instances is available at the homepage cited above.
	\item[Hardware and System] We used a \verb+Intel(R) Core(TM) i7-6700 CPU+ with 3.4 GHz, 4 cores, 8 threads and 16 GB of RAM under openSUSE Leap 42.3 for our computations.
	\item[Stopping Criteria] 
		For our own methods in \textsc{Python} we used a tolerance of $\varepsilon=10^{-7}$.
		For all computations in \textsc{Matlab} we used the default accuracy $\varepsilon=1.49\text{e-8}$ and $1.22\text{e-4}$ for inaccurate solutions.
		To avoid memory errors, we call SOS only if the size of the matrix lies below a certain bound.
		For Python we check $\binom{n+d}{d}\leq120$ and for Matlab $\binom{n+d}{d}\leq400$.
		If $400<\binom{n+d}{d}\leq 1000$, we only call SOSTOOLS with the ``sparse'' option.
		To avoid excessive run times, we call call SONC with variable split in Matlab only if $t(n+1)<3000$.
		We obtained all of these thresholds experimentally.
	\item[Runtime and Memory]
		The overall running time for all our instances was 28 days and created about 23 GB of data, mainly consisting of the certificates of nonnegativity.
\end{description}

\subsection{Generating Polynomials}
\label{subsection:generating_polynomials}

In what follows we describe how we generated our examples. As parameters for the input size we define by
\begin{description}
	\item[\struc{$n$}] the number of variables, taking values $n=2,3,4,8,10,20,30,40$,
	\item[\struc{$d$}] the degree, taking values $d=6,8,10,20,30,40,50,60$,
	\item[\struc{$t$}] the number of monomials, taking values $t=6,9,12,20,24,30,50,100,200,300,500$,
	\item[\struc{$\variableStyle{inner}$}] a lower bound for the number of summands whose exponent is not a vertex of the Newton polytope (only for the case ``arbitrary'' below). 
		%Note that $\variableStyle{inner} \leq |\support{p}\setminus\vertices{p}|$.
\end{description}

We created random polynomials of the following three classes.
\begin{description}
	\item[standard simplex] We have $\vertices{p} = \{\Vector 0, d\mathbf{e}_1,\ldots,d\mathbf{e}_n\}$ and $t-n-1$ many elements in $\interior{p}$, which we choose randomly using an even distribution over the lattice points of the interior of the scaled standard simplex. 
	\item[simplex] \label{item:generate_simplex}
		We put $\Vector{v}_0=\Vector 0$ and randomly choose by even distribution $n$ points from the scaled standard simplex $\standardSimplex{d/2}{n}$ and double their entries.
		These points form the vertex set $\vertices{p} = \{\Vector{v}_0,\Vector{v}_1,\ldots,\Vector{v}_n\}$.
		Afterwards we choose a random $\Vector{\lambda}\in[0,1]^{n+1}$ and normalize it to $| \Vector{\lambda}|_1 = \sum_{i = 0}^n \lambda_i = 1$.
		Then we compute $\Vector{v}$ as element wise rounding of $\left(\sum_{i=0}^n \lambda_i v_i\right)$ and check whether $\Vector{v}\in\operatorname{int}(\operatorname{conv}\{v_0,\ldots,v_n\})$.
		If that is the case, then we add $\Vector{v}$ to the set $\support{p}$.
		We iterate this process until we reach $\#\support{p}=t$ or some threshold of iterations (in which case the generation fails).
	\item[arbitrary] 
		We put $\Vector{v}_0=\Vector 0$ and randomly choose by even distribution $t-\variableStyle{inner}-1$ points from the scaled standard simplex $\standardSimplex{d/2}{n}$ and double their entries.
		All further vertices are chosen as in the simplex case, only as convex combination of $\{\Vector{v}_0,\Vector{v}_1,\ldots,\Vector{v}_{t-\variableStyle{inner}-1}\}$ (and thus may fail as well, if does not finish after a certain number of iterations).
		This way, we have ensured, that we have at least $\variableStyle{inner}$ points in $\support{p}\setminus\vertices{p}$..
\end{description}

We compute the convex hull of the set of chosen points (trivial in both simplex cases).
Let $h$ be the number of vertices of the hull.

Finally, we create an array of length $h$, normally distributed from $\mathcal N\left(0,\left(\frac{t}{n}\right)^2\right)$, take the absolute value of each entry, and concatenate this with an array of length $t-h$ normally distributed from $\mathcal N\left(0,1\right)$.
This is the coefficient vector $\Vector{b}$.
We chose standard deviation $\frac{t}{n}$ to keep the lower bounds in a reasonable range.

For each combination of parameters $(n,d,t)$ we ran the procedure with 10 different seeds and all three of the above shapes.
For the case ``arbitrary'', we additionally take $\variableStyle{inner}=\frac{k}{5}(t-n-1)$ for $k=1,2,3,4$.
In the end, we created \instances{} instances.

The limit of our implementation lies in the range, the random number generator can handle.
In our setup, the methods ``simplex'' and ``arbitrary'' fail for $n=40$ and $d=60$.
To pick evenly distributed vectors, we choose random numbers in the range of $0$ to $\binom{40+60}{40}$, which becomes too large for the underlying C data type to handle.

\subsection{Discussion of Selected Examples}
\label{subsection:selected_examples}

Before we present the outcome of the entire experiment, we discuss a couple of chosen examples.
In each table, ``\struc{opt}'' denotes the optimal value \GPopt{} or \SOSopt{}, according to the chosen strategy SONC or SOS respectively.

Since the polynomials discussed in the \Cref{example:standard_simplex,example:general_newton,example:dwarfed_cube} are very huge, we do not explicitly state them in the article. They are available online via

\begin{center}
	\url{https://www3.math.tu-berlin.de/combi/RAAGConOpt/comparison_paper/}
\end{center}

\begin{example}[Improved bounds with \cref{algorithm:bound}]
	We start with Example 5.5 from \cite{Dressler:Iliman:deWolff:SONCApproachConstraints}.
		\begin{align*}
			p = 1 + 3\cdot x_0^{2} x_1^{6} + 2\cdot x_0^{6} x_1^{2} + 6\cdot x_0^{2} x_1^{2} - 1\cdot x_0^{1} x_1^{2} - 2\cdot x_0^{2} x_1^{1} - 3\cdot x_0^{3} x_1^{3}
		\end{align*}
		In this paper, the authors achieve a lower bound $0.5732$ for $p$ using SONC with even split of coefficients and a bound $0
.6583$ with a different, arbitrarily chosen split.
		Using the same cover, as given in the paper, our \cref{algorithm:bound} yields the optimum $\variableStyle{opt} = -0.693158$, so we have $0.693158$ as an improved lower bound for $p$ given by SONC.
		\label{example:DIdW16_5.5}
\end{example}

\begin{example}[Range of runtimes of solvers]
	We consider a polynomial whose Newton polytope is a standard simplex with $n=10$, $d=30$, $t=200$.
	Running time and results are shown in \cref{table:example_standard_simplex}.
	All three solvers arrive at the same optimum, and ECOS is by far the fastest method.
	Due to the problem size, we did not attempt to compute a bound using SOS.
	\begin{table}[t]
		\begin{tabular}{lllrr}
		\hline
		language   & strategy   & solver   &   time (s) &     opt \\
		\hline
		python     & sonc       & ECOS     &   0.41     & 1109.45 \\
		matlab     & sonc       & sedumi   &  27.24     & 1109.45 \\
		matlab     & sonc       & sdpt3    & 153.09     & 1109.45 \\
%		python     & sonc       & SCS      & 181.91     & 2276.69 \\
		\hline
		\end{tabular}
		\caption{Standard simplex Newton polytope, $n=10$, $d=30$, $t=200$}
		\label{table:example_standard_simplex}
		%p1 = Polynomial('standard_simplex', 10,30,200,seed = 0)
	\end{table}
	\label{example:standard_simplex}
\end{example}

\begin{example}[Numerically stable via SONC/GP but unstable via SOS/SDP]
	We consider the following randomly generated polynomial, whose Newton polytope is a non-standard simplex, and which satisfies $n=5$, $d=8$, $t=10$.
	{\footnotesize
		\begin{align*}
			p \ = \ &1.9450715782850738 + 4.267736494409075\cdot x_3^{2} x_4^{2} + 0.8128309873524124\cdot x_2^{8} \\
			&+ 0.3863534030996798\cdot x_1^{4} x_2^{2} x_3^{2} + 1.5114805777890852\cdot x_1^{6} x_4^{2} + 1.07826527350598\cdot x_0^{6} x_1^{2} \\
			&- 0.7496903447028966\cdot x_0^{1} x_1^{2} x_2^{1} x_3^{1} x_4^{1} + 0.0328087476137118\cdot x_0^{1} x_1^{3} x_2^{1} x_3^{1} x_4^{1} \\
			&- 2.5827966329699446\cdot x_0^{1} x_1^{1} x_2^{2} x_3^{1} x_4^{1} - 1.1539503636520094\cdot x_0^{2} x_1^{1} x_2^{1} x_3^{1} x_4^{1}
		\end{align*}
	}
	The results of our computation are shown in \cref{table:example_simplex}.
	For SOS, numerical issues occur. Only \textsc{Sedumi}, \textsc{Gloptipoly}, and \textsc{SOSTOOLS} returned a solution.
	Furthermore, \textsc{Gloptipoly}'s solution violates the constraints by more than $\varepsilon=10^{-7}$, as denoted by $-1$ in the column ``\struc{verify}''.
	\textsc{SOSTOOLS} returned an answer, which is far higher than the other two SOS bounds, likely caused by a numerical instability.
	But the main observation is that SONC does not have any numerical issues and (therefore) yields a dramatically better bound than SOS, which we could obtain with three different solvers.
	Again, ECOS is by far the fastest solver.
	\begin{table}[t]
		\begin{tabular}{lllrrr}
			\hline
			language   & strategy   & solver             &   time (s) &        opt &   verify \\
			\hline
			matlab     & sos        & sostools, sparse   &   0.23     &  inf       &       -1 \\
			matlab     & sos        & sostools           &   7.94     & 4633.91    &        1 \\
			matlab     & sos        & yalmip             &   3.22     &  inf       &       -1 \\
			matlab     & sos        & gloptipoly         &   9.91     &   68.4777  &       -1 \\
			matlab     & sos        & sedumi             &  19.70     &   52.1625  &        1 \\
			matlab     & sos        & sdpt3              &  18.62     &  nan       &       -1 \\
			python     & sos        & CVXOPT             & 998.47     &  inf       &       -1 \\
			\hline
			matlab     & sonc       & sdpt3              &   2.12     &   -4.24914 &        1 \\
			matlab     & sonc       & sedumi             &   0.99     &   -4.24914 &        1 \\
			python     & sonc       & ECOS               &   0.01     &   -4.24914 &        1 \\
			\hline
		\end{tabular}
		\caption{\cref{example:simplex}: The Newton polytope is a simplex; parameters: $n=5$, $d=8$, $t=10$.}
		\label{table:example_simplex}
		%p = Polynomial('simplex',5,8,10,seed = 0)
	\end{table}
	\label{example:simplex}
\end{example}

\begin{example}[A huge polynomial with non-simplex Newton polytope]
	\label{example:general_newton}
	As in \cref{example:standard_simplex} we consider a polynomial with $n=10$, $d=30$, $t=200$. Here, however, the Newton polytope is no simplex but has an arbitrary shape.
	We have the additional parameter $\variableStyle{inner} = 100$; i.e., we have $100$ monomial squares and $100$ possibly negative terms.
	This increases the difficulty of the problem, since we first have to compute a cover of the Newton polytope with simplices, as described in \cref{algorithm:cover}.
	We present our results in \cref{table:example_general}.
	In \textsc{Matlab}, splitting the coefficients evenly and solving many instances of the simplex case greatly reduces the running time.
	For Python the difference is much smaller.
	As expected, for both solvers the variable split yields a better result.
	\begin{table}[t]
		\begin{tabular}{llllrr}
			\hline
			language	 & strategy   & solver  & splitting &  time (s) &      opt \\
			\hline
			matlab	   & sonc       & sedumi  & even      &  25.46    & -16.4538 \\
			matlab	   & sonc       & sedumi  & variable  & 124.69    & -16.535  \\
			python	   & sonc       & ECOS    & even      &   3.88    & -16.4538 \\
			python	   & sonc       & ECOS    & variable  &   4.79    & -16.535  \\
			\hline
		\end{tabular}
		\caption{Example \ref{example:general_newton}: Polynomial with general Newton polytope, $n=10$, $d=30$, $t=200$, $\variableStyle{inner} = 100$}
		\label{table:example_general}
		%p = Polynomial('general', 10, 30, 200, 100, seed = 8)
	\end{table}
\end{example}

\begin{example}[A polynomial with combinatorially challenging Newton polytope]
	\label{example:dwarfed_cube}
	The \struc{\textit{dwarfed cube}} is a polytope, that exhibits problematic behavior from a combinatorial point of view, in the sense, that it causes many convex hull algorithms to perform badly \cite{Avis:Bremner:Seidel:ConvexHullAlgorithms}.
	So we want to see, how our algorithms behave on a polynomial whose Newton polytope is the dwarfed cube.
	We choose a full support for the polynomial in the sense that every lattice point of the dwarfed cube occurs in the support.

	We show the results in \cref{table:example:dwarfed_cube}.
	SOS and SONC differ by less than $1\%$, but again SONC with ECOS is the fastest method, this time by a factor 4 compared to the fastest SOS-method YALMIP.
	\begin{table}[t]
		\begin{tabular}{llllrrr}
			\hline
			language	 & strategy   & solver           & splitting &   time (s)&      opt\\% &   verify \\
			\hline
			matlab	   & sos        & gloptipoly       &           &  8.76     & -28.3181\\% &        0 \\
			matlab	   & sos        & yalmip           &           &  4.02     & -28.3181\\% &        1 \\
			matlab	   & sos        & sostools, sparse &           &  7.98     & -28.3181\\% &        1 \\
			matlab	   & sos        & sostools         &           &  8.27     & -28.3181\\% &        1 \\
			matlab	   & sos        & sdpt3            &           & 11.87     & -28.3181\\% &        1 \\
			matlab	   & sos        & sedumi           &           & 17.69     & -28.3181\\% &        1 \\
			python	   & sos        & CVXOPT           &           &317.85     & -28.3181\\% &        1 \\
			\hline
			matlab	   & sonc       & sedumi           & evenly    &  4.16     & -28.246 \\% &        0 \\
			python	   & sonc       & ECOS             & evenly    &  1.07     & -28.246 \\% &        1 \\
			python	   & sonc       & ECOS             & variable  &  0.90     & -28.2777\\% &        1 \\
			\hline
		\end{tabular}
		\caption{Example \ref{example:dwarfed_cube}: Polynomial supported over the full 7-dimensional dwarfed cube, scaled by a factor 4, as support. Corresponding parameters: $n=7$, $d=6$, $t=113$, $\variableStyle{inner} = 63$}
		\label{table:example:dwarfed_cube}
	\end{table}
\end{example}

\subsection{Evaluation of the Experiment}
\label{subsection:evaluation}

In this section we present and evaluate the results of our experiment. 

\begin{description}
	\item[The running time of SONC is independent of the degree] 
		This fact was observed in examples in \cite{Dressler:Iliman:deWolff:SONCApproachConstraints} (see also further reference there in).
		We proved it more formally in \cref{theorem:degree_independent}, but it left open the possibility that the degree influences the number of iterations in the LPs or GPs.
		Here, we confirm experimentally that the running time for SONC is practically not affected by the degree.
		We investigate our test cases of arbitrary Newton polytope with $n=4$, $t=20$, using \cref{problem:sonc_var_split} solved by ECOS.
		The results are shown in \cref{table:degree_independent}.
		\begin{table}
			\begin{tabular}{ccrc|c|ccrc}
				degree & inner\_terms& time & instances &&
				degree & inner\_terms& time & instances \\
				\hline
				40	   & 3          & 0.100& 10        &&6      & 3          & 0.138& 9         \\
				50	   & 3          & 0.101& 10        &&8      & 3          & 0.140& 9         \\
				30	   & 3          & 0.105& 10        &&10     & 6          & 0.140& 10        \\
				60	   & 3          & 0.106& 10        &&10     & 3          & 0.140& 10        \\
				60	   & 9          & 0.113& 10        &&40     & 12         & 0.142& 10        \\
				50	   & 6          & 0.115& 10        &&50     & 12         & 0.143& 10        \\
				20	   & 3          & 0.115& 10        &&30     & 12         & 0.144& 10        \\
				60	   & 6          & 0.115& 10        &&10     & 9          & 0.152& 10        \\
				30	   & 6          & 0.116& 10        &&8      & 6          & 0.153& 10        \\
				50	   & 9          & 0.120& 10        &&60     & 12         & 0.153& 10        \\
				20	   & 6          & 0.124& 10        &&10     & 12         & 0.155& 10        \\
				20	   & 9          & 0.127& 10        &&6      & 6          & 0.159& 8         \\
				40	   & 6          & 0.128& 10        &&8      & 9          & 0.164& 10        \\
				30	   & 9          & 0.134& 10        &&6      & 9          & 0.175& 8         \\
				40	   & 9          & 0.134& 10        &&8      & 12         & 0.181& 10        \\
				20	   & 12         & 0.137& 10        &&6      & 12         & 0.205& 8         \\
			\end{tabular}
			\caption{Timing for $n=4$, $t=20$, arbitrary Newton polytope, using ECOS, ordered by time.
			We see that the running time is independent of the degree.}
			\label{table:degree_independent}
		\end{table}
		Hence, for the following observations, we mainly regard dependencies on the number of variables and terms, and the shape of the Newton polytope.
	\item[SONC behavior on the standard simplex] We show how SONC behaves on the scaled standard simplex, which is the corresponding Newton polytope in the common approach to investigate (dense) polynomials in $n$ variables of degree $d$.
		The results are in presented \cref{table:standard_simplex}.
		Even for instances as large as 40 variables and 500 terms, we can solve the problem in a few seconds. 
		Note that we only consider instances with $d > n$, since otherwise the interior of the Newton polytope is empty.
		\begin{table}[t]
		\begin{tabular}{c|cccccccc}
			$t\setminus n$&2&3&4&8&10&20&30&40\\
			\hline
			6&0.054&0.048&0.048&-&-&-&-&-\\
			9&0.07&0.066&0.063&-&-&-&-&-\\
			12&0.088&0.083&0.084&0.071&0.116&-&-&-\\
			20&0.137&0.136&0.138&0.135&0.133&-&-&-\\
			24&0.162&0.163&0.16&0.158&0.158&0.132&-&-\\
			30&0.199&0.197&0.2&0.207&0.211&0.217&-&-\\
			50&0.307&0.326&0.338&0.365&0.384&0.507&0.508&0.401\\
			100&0.609&0.653&0.698&0.826&0.867&1.254&1.662&1.946\\
			200&1.262&1.373&1.526&1.834&1.991&3.06&4.2&5.471\\
			300&2.008&2.199&2.392&2.986&3.27&5.213&7.49&9.716\\
			500&3.795&4.107&4.558&5.858&6.579&11.177&17.018&20.889\\
		\end{tabular}
			\caption{Average running time for SONC with ECOS, where the Newton polytope is the standard simplex; depending on number of terms $t$ and number of variables $n$.
			A ``-'' indicates that no such instance exists.}
			\label{table:standard_simplex}
		\end{table}
	\item[A comparison of splitting strategies for SONC] We compare the strategy to split coefficients variably as described in \Cref{subsubsection:variable_split} with the strategy to split coefficients evenly as described in \Cref{subsubsection:even_split} and as it was done in \cite{Dressler:Iliman:deWolff:SONCApproachConstraints}.
		First note, whenever both strategies found a solution, then the variable split is at least as good as the even split.
		However, there are 421 cases, where the even split found a solution, but the variable split failed; probably the increased problem size leads to numerical issues.

		Using \textsc{ECOS}, we cannot verify our initial hypothesis regarding a speed improvement gained by using even split.
		The quotients of the running time of even splitting divided by the running time of variable splitting range from 0.03 to 3.07, but with an average 1.284.
		A closer inspection shows that only 1042 instances are solved faster by even splitting, whereas for 7893 instances the variable split is faster.
		If we distinguish this behavior by the number of terms and the number of variables, then we get the results from \cref{table:split_strategy_time}.
		Only for polynomials with many terms, but few variables, we get ratios below 1. 
		So only in these cases, the even splitting is faster on average.
		Since the bound obtained by the even splitting can be as large as to have an overflow, i.e. numerically infinite, we therefore suggest to use the variable splitting \textbf{always} as first choice.
		\begin{table}[t]
			\begin{tabular}{c|cccccccc}
				&2&3&4&8&10&20&30&40\\
				\hline
				6&1.08&1.07&1.1&-&-&-&-&-\\
				9&1.16&1.15&1.08&-&-&-&-&-\\
				12&1.22&1.2&1.21&1.04&-&-&-&-\\
				20&1.28&1.22&1.26&1.38&1.35&-&-&-\\
				24&1.26&1.2&1.23&1.42&1.51&1.03&-&-\\
				30&1.2&1.17&1.23&1.47&1.55&1.02&-&-\\
				50&1.07&1.07&1.13&1.43&1.6&2.04&1.42&-\\
				100&0.896&0.789&0.838&1.16&1.36&2.15&2.18&2.07\\
				200&-&0.389&0.317&0.609&0.791&1.94&2.11&1.97\\
				300&-&0.327&0.172&0.254&0.437&1.67&1.96&1.99\\
				500&-&-&0.0522&0.0694&0.134&1.14&1.76&1.85\\
			\end{tabular}
			\caption{Median quotient of the time ``even'' over ``variable'' splitting.}
			\label{table:split_strategy_time}
		\end{table}
	\item[Handling degenerate terms]
		In our whole data set, \degenerateInstances{} instances contain degenerate terms.
		For \unboundedInstances{} of these, we could certify, that they are unbounded.
		In \soncBoundDegenerate{} cases, which means about \soncDegeneratePercent{}\%, we could compute a bound via SONC, despite the presence of degenerate points.
		So we conclude experimentally, that degenerate terms \emph{might} cause problems for SONC, but \emph{usually} we still can compute a bound.
	\item[Qualitative runtime comparison of SOS and SONC] 
		\newcommand{\noTry}{$\times$}
		\newcommand{\specialTry}[1]{#1$^*$}
		In \cref{table:time_SOS}, in the first lines of each cell, we show the average running time of the fastest one among the SOS methods which we ran on each of our instances.
		Already for 8 variables in degree 6 we observe a drastic increase in the running time.
		Furthermore, in each entry marked with a ``\noTry'', the dimension of the psd-matrix in the corresponding SDP would have exceeded $1000\times1000$.
		During our experiments we observed that problems beyond this bound regularly lead to memory errors, because they require more than 16GB RAM.
		Even if they do not, we easily have run times of several hours, despite using the ``sparse'' option of SOSTOOLS.
		Thus, we did not attempt solving these problems systematically.
		\begin{table}[t]
		\begin{tabular}{c|cccccccc}
			$d\setminus n$&2&3&4&8&10&20&30&40\\
			\hline
			\multirow{2}{*}{6}&0.147&0.215&0.275&40.345&542.595&\noTry&-&-\\
			&\textbf{0.0929}&\textbf{0.169}&\textbf{0.259}&\textbf{1.34}&\textbf{15.5}&\textbf{0.225}&-&-\\
			\hline
			\multirow{2}{*}{8}&0.179&0.296&0.642&\specialTry{3.236}\tablefootnote{Only a single instance was solved by SOS.}&\noTry&-&-&-\\
			&\textbf{0.14}&\textbf{0.217}&\textbf{0.39}&\textbf{1.33}&\textbf{0.345}&-&-&-\\
			\hline
			\multirow{2}{*}{10}&0.202&0.365&2.095&\noTry&\noTry&\noTry&-&-\\
			&\textbf{0.138}&\textbf{0.23}&\textbf{0.474}&\textbf{0.45}&\textbf{0.381}&\textbf{12.6}&-&-\\
			\hline
			\multirow{2}{*}{20}&0.357&9.25&\specialTry{3074}\tablefootnote{Computations were aborted after few instances.}&\noTry&\noTry&\noTry&-&-\\
			&\textbf{0.199}&\textbf{0.4}&\textbf{1.08}&\textbf{7.28}&\textbf{10.7}&\textbf{0.969}&-&-\\
			\hline
			\multirow{2}{*}{30}&0.966&\specialTry{414.45}&\noTry&\noTry&\noTry&\noTry&-&-\\
			&\textbf{0.337}&\textbf{0.982}&\textbf{4.5}&\textbf{19.5}&\textbf{103}&\textbf{25.6}&-&-\\
			\hline
			\multirow{2}{*}{40}&3.445&\noTry&\noTry&\noTry&\noTry&\noTry&\noTry&-\\
			&\textbf{0.82}&\textbf{2.65}&\textbf{5.33}&\textbf{108}&\textbf{93.1}&\textbf{38.7}&\textbf{23.2}&-\\
			\hline
			\multirow{2}{*}{50}&11.583&\noTry&\noTry&\noTry&\noTry&\noTry&\noTry&\noTry\\
			&\textbf{0.721}&\textbf{4.04}&\textbf{17.1}&\textbf{154}&\textbf{83.1}&\textbf{25.2}&\textbf{27.1}&\textbf{41.9}\\
			\hline
			\multirow{2}{*}{60}&2.277&\noTry&\noTry&\noTry&\noTry&\noTry&\noTry&-\\
			&\textbf{0.766}&\textbf{7.98}&\textbf{170}&\textbf{148}&\textbf{95.4}&\textbf{33}&\textbf{36.1}&-\\
		\end{tabular}
			\caption{Average running time of the fastest SOS method (normal font) compared to ECOS with variable split (bold), depending on degree $d$ and number of variables $n$.
			Note that for higher degree, we generally have more terms.
			A ``\noTry'' indicates, that solving was not attempted, to avoid memory errors.
			A ``-'' indicates, we have no solved instance with these parameters.
			A ``\specialTry{}'' means, we only attempted SOSTOOLS with the sparse option.
			}
			\label{table:time_SOS}
			\label{table:time_SONC}
		\end{table}
		We already observed that the running time of SONC depends on the number of terms, but not on the degree. 
		The running time of SOS, however, depends on the degree primarily, since, in contrast to SONC, computing an SOS certificate does not keep the support invariant. 
		In order to compare both methods we hence proceed as follows:
		For each pair $(n,d)$ we consider the largest number of terms $t$, where we created instances, and compute the average running time of SONC only for those instances.
		Thus, we have a worst-case scenario for SONC in our setting.
		The results are presented in \cref{table:time_SONC} in the second lines.
		Note that for large degrees we can also expect a large Newton polytopes and hence can have more terms in the instances. 
		Thus, we consider our setup to be disadvantageous for SONC especially in these cases. 
		The ``-'' in the table mean, we do not have instances with arbitrary Newton polytope for these parameters.
		For $(n,d)=(40,60)$, we met the limit of our method for generating polynomials, which we explained in \Cref{subsection:generating_polynomials}.
		In the other cases, our generation algorithm failed to find enough terms.

		We make the following two main observations:
		\begin{itemize}
			\item We can solve far larger and hence far more instances with SONC, than we could solve with SOS.
			\item Even in the parameter range, where SOS found a solution, SONC usually is faster (despite the disadvantage for SONC in the setup).
		\end{itemize}

		To exemplify this difference, we take a close look at the instances satisfying $n=3$ and $d=20$. We evaluate the average runtimes for varying numbers of terms, as shown in \cref{figure:plot_n3_d20}. 
		The running time of SONC starts at 61ms for $t=6$ and reaches 3.56s for $t=500$.
		The growth of the runtime is roughly linear in $t$.
		On contrast, SOS requires more than 5s, even for few terms, going up to around 25s.
		Additionally, we observe, that SOS already exploited sparsity to some extent; for fewer terms, we can expect a smaller Newton polytope, which results in smaller SDPs.
		\begin{figure}[t]
			\ifpictures
			\centering
			\input{plot_n3_d20.tex}
			\caption{Average running time of SONC and SOS for instances with $n=3$ and $d=20$, depending on the number of terms.}
			\label{figure:plot_n3_d20}
			\fi
		\end{figure}
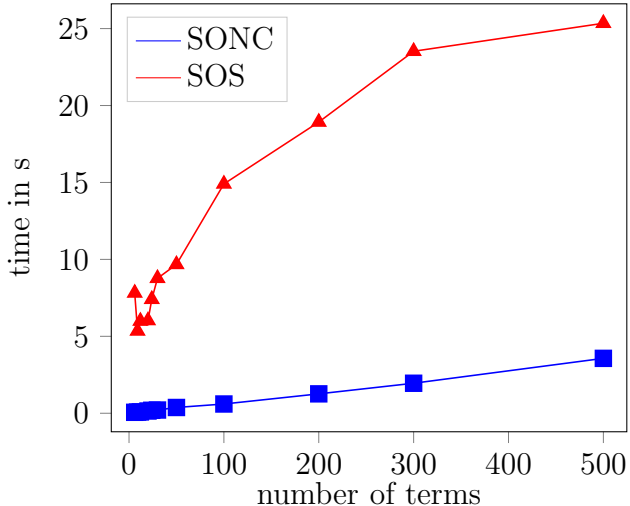
	\item[Quantitative comparison of SONC and SOS] For every single instance, we could solve with SONC, \textsc{ECOS} was the fastest method, unless it failed to find a solution at all.
		%select * from best where name_sonc != 'ECOS' or time_sonc > time_sos;
		The latter happened in only 195 instances.
		Since such a failure, however, only takes few seconds or even milliseconds, a user can quickly switch to another method like \textsc{Matlab}/\textsc{SeDuMi}. 
		Overall, among our \instances{} instances, we found a SONC bound for \soncBound{} of them.
		Only in 4 cases, the fastest SOS method found a bound faster than SONC.
		In all of these, \textsc{ECOS} failed, so we found the SONC bound with \textsc{Matlab}/\textsc{SeDuMi}, which is considerably slower.
		Even if we restrict ourselves to SONC with variable splitting of the coefficients, there are still only 15 instances where SOS was quicker in finding a solution. 
		Since these were of different sizes, we could not find a pattern among these 15 instances.
		The largest difference and the best ratio for SOS regarding the running times was 5.6 seconds to 33.3 seconds. 

		An overview of the ratios of the running times between SOS and SONC is shown in the right graphic of \cref{figure:plot_bar_difference}.
		In about half of the cases, where both algorithms found a solution, SONC was faster by a factor of at least 4.
		However, there are instances with a speedup of more than 1000, and the largest ratio is around 20,000.

		In the left graphic of \cref{figure:plot_bar_difference} we list the differences between the optima, obtained by SONC and SOS.
		When both algorithms found a bound, then in most cases SOS found the better bound, but sometimes the bound of SONC was better.
		But the main observation is that there were only 44 instances, where SOS found a bound and SONC failed.
		Contrary to this, we have 11760 polynomial where SONC succeeded but SOS could not find a bound, or would have exceeded our thresholds, see \Cref{subsection:experimental:setup}.
		Of the 3693 instances, where SOS found a better bound than SONC, 782 instances (21.2\%) have the scaled standard simplex as Newton polytope.
		Since the scaled standard simplex is an \emph{\struc{$H$-simplex}}, we have by construction that $\SOSopt\leq \GPopt$ in these cases, see \cite[Section 5]{Iliman:deWolff:Circuits}.

		In addition, there are \trivialInstances{} instances, which are sums of monomial squares and 136 instances, where both algorithms failed.
		Out of these at least \unboundedInstances{} instances are unbounded.
		For the remaining 99 the status is unknown, since checking whether a polynomial is bounded is \coNP{}-hard on its own, see \cref{subsection:SOS}.
\end{description}

\section{Resume and Outlook}
\label{section:conclusion}
This article contains of two main contributions. First, we present a new algorithm, including implementation, to compute lower bounds for multivariate polynomials using SONC certificates.
This is the first implementation of this kind and it is the first algorithm to compute such a bound based on nonnegative circuit polynomials, where after giving the polynomial, even for arbitrary Newton polytope, no further human interaction is required.
Furthermore, this algorithm computes a lower bound for a sparse polynomial in polynomial time.

Second, we provide an experimental comparison of SONC (using our algorithm and GP) solvers and SOS (using various standard SDP solvers). Based on these experiments, we draw the following key conclusions:
\begin{enumerate}
	\item The runtime for SONC certificates depends on the number of variables and the cardinality of the support only. 
	It does not depend on the degree.
	\item Among the two strategies for splitting coefficients introduced in \Cref{subsection:arbitrary_newton}, the variable splitting strategy yields always better bounds and is, to our surprise, mostly faster than the even splitting strategy.
	\item Among the tested GP solvers, \textsc{ECOS} is \textbf{by far} the best one for computing SONC certificates. We remark, however, that \textsc{Mosek} announced to add an exponential cone solver in its next release, which we have not tested yet.
	\item SONC certificates can be computed extremely fast. Even for very large instances, the runtime remains in the realm of seconds. In a direct comparison of those cases where both SONC and SOS yield a bound, SONC is faster than SOS in over $\mathbf{99.5\%}$ of all cases.
	\item If both strategies yield a bound, then SOS yields better bounds than the current SONC algorithm in the clear majority ($82.5\%$) of the cases. 
		However, often the bounds do not differ very much, and there is a nontrivial amount of cases, when the SONC bound is better.
		See \cref{figure:plot_bar_difference} for details.
	\item SONC is extremely robust.
		Even for the majority of instances with degenerate terms, a potential source of problems for SONC, we were able to compute bounds.
	Moreover it uses vastly less memory than SOS. 
	While in our experiments (for 16 GB RAM) SOS certificates can be computed up to corresponding SDPs with matrices of size roughly $1000 \times 1000$  (these are e.g., instances of type $n=4$, $d=20$), SONC easily handles cases of 40 variables and 500 terms.
\end{enumerate}

For unconstrained polynomial optimization, we summarize our findings as follows:
\begin{enumerate}
	\item In a standard setting, it is, due to the short runtime and the little memory requirements, almost always best to compute a SONC certificate first, and then, depending on the size of the problem, try to compute an SOS certificate afterwards.
	\item If an SOS certificate is computable, then often, but not always, SOS yields better bounds than SONC.
	\item Especially for large, sparse instances, SONC is the certificate of choice.
	\item Large instances with many terms in the boundary are problematic for both SOS and SONC.
\end{enumerate}

In its current stage, SONC certificates are not ready for most applications, yet.
The main problems to overcome are, in our opinion:
\begin{enumerate}
	\item Improve the SONC theory of constrained optimization introduced in \cite{Dressler:Iliman:deWolff:SONCApproachConstraints,Dressler:Iliman:deWolff:Positivstellensatz}, relate it to Chandrasekaran's and Shah's SAGE certificates \cite{Chandrasekaran:Shah:SAGE}, and implement constrained optimization in \textsc{POEM}.
	\item Develop and implement a good way to compute minimizers exploiting a given SONC decomposition in order to obtain upper bounds and to compute the duality gap.
	\item Improve our algorithms (e.g., with respect to finding optimal coverings or carrying out preprocessing steps) to achieve better lower bounds via SONC.
\end{enumerate}

We will tackle these issues in future publications. 
The main challenge here is how this approach can be refined without losing polynomial running time in theory, and without vastly increasing the runtime and the required memory in practice.

\bibliographystyle{halphainit}
\bibliography{Experimental_Comparison_SONC_SOS_Arxiv1}

\end{document}

%% file: numbers.tex
\newcommand{\unboundedInstances}{37}
\newcommand{\soncBoundDegenerate}{2279}
\newcommand{\soncDegeneratePercent}{94}
\newcommand{\soncBound}{16240}
\newcommand{\degenerateInstances}{2413}
\newcommand{\trivialInstances}{501}
\newcommand{\instances}{16921}

%% file: plot_bar.tex
% This file was created by matplotlib2tikz v0.6.17.
\begin{tikzpicture}

\definecolor{color0}{rgb}{0.12156862745098,0.466666666666667,0.705882352941177}

\begin{groupplot}[group style={group size=2 by 1}]
\nextgroupplot[
ylabel={number of instances},
xmin=-0.74, xmax=6.74,
ymin=0, ymax=12348,
xtick={0.5,1.5,2.5,3.5,4.5,5.5,6.5},
xticklabels={$\infty$,$1$,$0.001$,$-0.001$,$-1$,$-\infty$,},
tick align=outside,
xticklabel style = {rotate=30},
tick pos=left,
x grid style={white!69.01960784313725!black},
y grid style={white!69.01960784313725!black}
]
\draw[fill=color0,draw opacity=0] (axis cs:-0.4,0) rectangle (axis cs:0.4,44);
\draw[fill=color0,draw opacity=0] (axis cs:0.6,0) rectangle (axis cs:1.4,1625);
\draw[fill=color0,draw opacity=0] (axis cs:1.6,0) rectangle (axis cs:2.4,2068);
\draw[fill=color0,draw opacity=0] (axis cs:2.6,0) rectangle (axis cs:3.4,598);
\draw[fill=color0,draw opacity=0] (axis cs:3.6,0) rectangle (axis cs:4.4,61);
\draw[fill=color0,draw opacity=0] (axis cs:4.6,0) rectangle (axis cs:5.4,127);
\draw[fill=color0,draw opacity=0] (axis cs:5.6,0) rectangle (axis cs:6.4,5732);
\draw[fill=red,draw opacity=0] (axis cs:-0.4,44) rectangle (axis cs:0.4,44);
\draw[fill=red,draw opacity=0] (axis cs:0.6,1625) rectangle (axis cs:1.4,1625);
\draw[fill=red,draw opacity=0] (axis cs:1.6,2068) rectangle (axis cs:2.4,2068);
\draw[fill=red,draw opacity=0] (axis cs:2.6,598) rectangle (axis cs:3.4,598);
\draw[fill=red,draw opacity=0] (axis cs:3.6,61) rectangle (axis cs:4.4,61);
\draw[fill=red,draw opacity=0] (axis cs:4.6,127) rectangle (axis cs:5.4,127);
\draw[fill=red,draw opacity=0] (axis cs:5.6,5732) rectangle (axis cs:6.4,11760);
\node at (axis cs:-0.35,144)[
  scale=0.5,
  anchor=base west,
  text=black,
  rotate=0.0
]{ 44};
\node at (axis cs:0.65,1725)[
  scale=0.5,
  anchor=base west,
  text=black,
  rotate=0.0
]{ 1625};
\node at (axis cs:1.65,2168)[
  scale=0.5,
  anchor=base west,
  text=black,
  rotate=0.0
]{ 2068};
\node at (axis cs:2.65,698)[
  scale=0.5,
  anchor=base west,
  text=black,
  rotate=0.0
]{ 598};
\node at (axis cs:3.65,161)[
  scale=0.5,
  anchor=base west,
  text=black,
  rotate=0.0
]{ 61};
\node at (axis cs:4.65,227)[
  scale=0.5,
  anchor=base west,
  text=black,
  rotate=0.0
]{ 127};
\node at (axis cs:5.65,5832)[
  scale=0.5,
  anchor=base west,
  text=black,
  rotate=0.0
]{ 5732};
\node at (axis cs:5.65,11860)[
  scale=0.5,
  anchor=base west,
  text=red,
  rotate=0.0
]{ 11760};
\nextgroupplot[
xmin=-0.84, xmax=8.84,
ymin=0, ymax=12348,
xtick={0.5,1.5,2.5,3.5,4.5,5.5,6.5,7.5,8.5},
xticklabels={$0$,$1$,$2$,$4$,$10$,$100$,$1000$,$\infty$,},
tick align=outside,
xticklabel style = {rotate=30},
tick pos=left,
x grid style={white!69.01960784313725!black},
y grid style={white!69.01960784313725!black}
]
\draw[fill=color0,draw opacity=0] (axis cs:-0.4,0) rectangle (axis cs:0.4,44);
\draw[fill=color0,draw opacity=0] (axis cs:0.6,0) rectangle (axis cs:1.4,15);
\draw[fill=color0,draw opacity=0] (axis cs:1.6,0) rectangle (axis cs:2.4,510);
\draw[fill=color0,draw opacity=0] (axis cs:2.6,0) rectangle (axis cs:3.4,1709);
\draw[fill=color0,draw opacity=0] (axis cs:3.6,0) rectangle (axis cs:4.4,876);
\draw[fill=color0,draw opacity=0] (axis cs:4.6,0) rectangle (axis cs:5.4,1071);
\draw[fill=color0,draw opacity=0] (axis cs:5.6,0) rectangle (axis cs:6.4,250);
\draw[fill=color0,draw opacity=0] (axis cs:6.6,0) rectangle (axis cs:7.4,32);
\draw[fill=color0,draw opacity=0] (axis cs:7.6,0) rectangle (axis cs:8.4,11760);
\draw[fill=red,draw opacity=0] (axis cs:-0.4,44) rectangle (axis cs:0.4,44);
\draw[fill=red,draw opacity=0] (axis cs:0.6,15) rectangle (axis cs:1.4,15);
\draw[fill=red,draw opacity=0] (axis cs:1.6,510) rectangle (axis cs:2.4,510);
\draw[fill=red,draw opacity=0] (axis cs:2.6,1709) rectangle (axis cs:3.4,1709);
\draw[fill=red,draw opacity=0] (axis cs:3.6,876) rectangle (axis cs:4.4,876);
\draw[fill=red,draw opacity=0] (axis cs:4.6,1071) rectangle (axis cs:5.4,1071);
\draw[fill=red,draw opacity=0] (axis cs:5.6,250) rectangle (axis cs:6.4,250);
\draw[fill=red,draw opacity=0] (axis cs:6.6,32) rectangle (axis cs:7.4,32);
\draw[fill=red,draw opacity=0] (axis cs:7.6,5732) rectangle (axis cs:8.4,11760);
\node at (axis cs:-0.35,144)[
  scale=0.5,
  anchor=base west,
  text=black,
  rotate=0.0
]{ 44};
\node at (axis cs:0.65,115)[
  scale=0.5,
  anchor=base west,
  text=black,
  rotate=0.0
]{ 15};
\node at (axis cs:1.65,610)[
  scale=0.5,
  anchor=base west,
  text=black,
  rotate=0.0
]{ 510};
\node at (axis cs:2.65,1809)[
  scale=0.5,
  anchor=base west,
  text=black,
  rotate=0.0
]{ 1709};
\node at (axis cs:3.65,976)[
  scale=0.5,
  anchor=base west,
  text=black,
  rotate=0.0
]{ 876};
\node at (axis cs:4.65,1171)[
  scale=0.5,
  anchor=base west,
  text=black,
  rotate=0.0
]{ 1071};
\node at (axis cs:5.65,350)[
  scale=0.5,
  anchor=base west,
  text=black,
  rotate=0.0
]{ 250};
\node at (axis cs:6.65,132)[
  scale=0.5,
  anchor=base west,
  text=black,
  rotate=0.0
]{ 32};
\node at (axis cs:7.65,5832)[
  scale=0.5,
  anchor=base west,
  text=black,
  rotate=0.0
]{ 5732};
\node at (axis cs:7.5,11860)[
  scale=0.5,
  anchor=base west,
  text=red,
  rotate=0.0
]{ 11760};
\end{groupplot}

\end{tikzpicture}

%% file: plot_n3_d20.tex
% This file was created by matplotlib2tikz v0.6.17.
\begin{tikzpicture}

\begin{axis}[
xlabel={number of terms},
ylabel={time in s},
xmin=-18.7, xmax=524.7,
ymin=-1.20321539833334, ymax=26.609912765,
tick align=outside,
tick pos=left,
x grid style={white!69.01960784313725!black},
y grid style={white!69.01960784313725!black},
legend entries={{SONC},{SOS}},
legend style={at={(0.03,0.97)}, anchor=north west, draw=white!80.0!black},
legend cell align={left}
]
\addlegendimage{no markers, blue}
\addlegendimage{no markers, red}
\addplot [semithick, blue]
table {%
6 0.0610177
9 0.075445896551
12 0.087069847457
20 0.145920833333
24 0.192931637931
30 0.209464666666
50 0.374968220338
100 0.60111368421
200 1.26031945
300 1.949864210526
500 3.567698583333
};
\addplot [semithick, blue, mark=square*, mark size=3, mark options={solid}, only marks, forget plot]
table {%
6 0.0610177
9 0.075445896551
12 0.087069847457
20 0.145920833333
24 0.192931637931
30 0.209464666666
50 0.374968220338
100 0.60111368421
200 1.26031945
300 1.949864210526
500 3.567698583333
};
\addplot [semithick, red]
table {%
6 7.81487353333333
9 5.34646184482758
12 5.99127408474576
20 6.02587753333333
24 7.40526253448276
30 8.76933466666667
50 9.67258416949153
100 14.9026767894737
200 18.92487615
300 23.5258575789474
500 25.3456796666667
};
\addplot [semithick, red, mark=triangle*, mark size=3, mark options={solid}, only marks, forget plot]
table {%
6 7.81487353333333
9 5.34646184482758
12 5.99127408474576
20 6.02587753333333
24 7.40526253448276
30 8.76933466666667
50 9.67258416949153
100 14.9026767894737
200 18.92487615
300 23.5258575789474
500 25.3456796666667
};
\end{axis}

\end{tikzpicture}